\documentclass[10pt]{amsart}
\usepackage{graphicx,epstopdf,xypic,somedefs,tracefnt,latexsym,rawfonts,amsfonts,amssymb,latexsym,enumerate,amscd}
\def\cal{\mathcal}
\def\Bbb{\mathbb}

\def\G{\Gamma}
\def\r{\rangle}
\def\l{\langle}

\newtheorem{prop}{Proposition}[section]
\newtheorem{thm}{Theorem}[section]

\newtheorem*{ack}{Acknowledgement}
\newtheorem{lemma}{Lemma}[section]

\newtheorem{cor}{Corollary}[section]

\newtheorem{defn}{Definition}[section]

\newtheorem{rem}{Remark}[section]
\numberwithin{equation}{section}
\begin{document}
\title[A certain structure of Artin groups]{A certain structure of Artin groups and the isomorphism 
conjecture}
\author[S.K. Roushon]{S.K. Roushon}
\address{School of Mathematics\\
Tata Institute\\
Homi Bhabha Road\\
Mumbai 400005, India}
\email{roushon@math.tifr.res.in} 
\urladdr{http://www.math.tifr.res.in/\~\ roushon/}
\thanks{May 06, 2020. To appear in the {\it Canadian Journal of Mathematics}}
\begin{abstract}
  We observe an inductive structure in a large class of
  Artin groups of finite real, complex and affine types and 
exploit this information to  
deduce the Farrell-Jones isomorphism conjecture for these groups.\end{abstract}


\maketitle


\section{Introduction}
We prove the   
isomorphism conjecture of Farrell and Jones (\cite{FJ}, \cite{lueck}) in
the $K$-, $L$- and $A$-theories, for a large class of Artin 
groups of finite real, complex and affine types. The classical
braid group is an example of a finite type real Artin group
(type $A_n$), and we considered this group in \cite{FR} and \cite{R5} in the 
pseudoisotopy case of the conjecture.

The classical braid group case as in \cite{FR} and \cite{R5}, 
used the crucial idea that,  
the geometry of a certain class of $3$-manifolds is involved in the building of the group. 
In the situation of the Artin groups considered in this paper,  
we find a similar structure. We exploit this information to prove the conjecture. More precisely, 
some of the Artin groups are realized as a subgroup of the orbifold 
fundamental group of the configuration space of unordered $n$-tuples of distinct  
points, on some $2$-dimensional orbifolds. In the classical braid group case, 
the complex plane played the role of this $2$-dimensional orbifold.

The isomorphism conjecture is an important conjecture  
in Geometry and Topology. Much work has been done in this area in recent times 
(e.g. \cite{BB}, \cite{BL}, \cite{Rou0}, \cite{Rou01}, \cite{R5}). 
The motivations behind the isomorphism conjecture are 
the Borel and the Novikov conjectures, which claim that 
any two closed aspherical homotopy equivalent manifolds are 
homeomorphic, and the homotopy invariance of higher signatures
of manifolds, respectively. Furthermore, the
isomorphism conjecture for a group provides a 
better understanding of the $K$- and $L$-theories of the group.
Also it is well-known that the isomorphism conjecture for a group implies vanishing 
of the lower $K$-theory ($Wh(-), \tilde K_0({\Bbb Z}[-])$ and
$K_{-i}({\Bbb Z}[-])$ for $i\geq 1$) 
of any torsion free subgroup of the group. In fact, it is still an open conjecture 
that this should be the case for all torsion free groups. A consequence
of the isomorphism conjecture is another conjecture by Hsiang, which says 
that $K_{-i}({\Bbb Z}[-])=0$ for all $i\geq 2$ and for all groups (\cite{Hsi}).

There are two classical 
exact sequences in $K$- and $L$-theories, involving the following two
assembly maps, where $R$ is a regular ring
in $A_K$, and is a ring with involution in $A_L$. This 
summarizes the history of the Borel and the Novikov conjectures.

$$A_K:H_*(B\G, {\Bbb K}(R))\to K_*(R[\G]).$$
$$A_L:H_*(B\G, {\Bbb L}(R))\to L_*(R[\G]).$$

In the case of torsion free groups, the isomorphism conjecture says
that the above two maps are isomorphisms. See [\cite{FJ}, \S1.6.1] or
[\cite{lueck}, Remark 20.12.3].

Following Farrell-Jones and Farrell-Hsiang, in all of recent works on the 
isomorphism conjecture, geometric input on the 
group plays a significant role. Here also we follow a similar path to prove the 
conjecture, but we need to use an inductive structure in the groups, where
the building blocks carry certain (non-positively curved) geometry. The existence of any global
geometric structure on Artin groups is not yet known. A problem is 
stated in \cite{RC}, asking if these groups should be $CAT(0)$, and for
$CAT(0)$ groups the isomorphism conjecture is known (\cite{BL}, \cite{Weg}).

We prove the following theorem.

\begin{thm} \label{fic-braid} Let $\G$ be an Artin 
group of type $A_n$, $B_n (=C_n)$, $D_n$, $F_4$, $G_2$, $I_2(p)$, $\tilde A_n$,  
$\tilde B_n$, $\tilde C_n$ or $G(de,e,r)$ ($d,r\geq 2$). Then the isomorphism 
conjecture in $K$-, $L$- and $A$-theories with coefficients and finite wreath product is true for any 
subgroup of $\G$. That is, the isomorphism conjecture with coefficients is true for 
$H\wr G$, for any subgroup $H$ of $\G$, and for any finite group $G$.\end{thm}

\begin{proof} In Section 3 we introduce a class $\cal C$ (Definition \ref{C}) of groups defined
  inductively, using suitable group extensions, from fundamental groups of manifold of
  dimension $\leq 3$. Then in Theorems \ref{orbi-poly-free} and \ref{d-b} we show that
  the groups considered in Theorem \ref{fic-braid} belong to $\cal C$. Finally, in 
  Theorem \ref{ficwf} we observe that the groups in $\cal C$ satisfy the Farrell-Jones
  isomorphism conjecture with coefficients and finite wreath product.\end{proof}

Since the groups considered in the theorem are torsion free,
a consequence of the theorem is that, the two assembly maps above are
isomorphisms for the subgroups of any of these groups.

 Let $\Gamma$ be a finite type 
 pure Artin group appearing in Theorem \ref{fic-braid}.
 Then as an application, in Theorem \ref{surgery}, we 
extend our earlier computation of the surgery groups of the
classical pure Artin braid groups (\cite{R5}) to $\Gamma$.

Below we mention a well known 
consequence of the isomorphism of the $K$-theory 
assembly map. See [\cite{FJ}, \S 1.6.1].

\begin{cor} \label{lkt} The Whitehead group $Wh(-)$, the reduced projective class 
group $\tilde K_0({\Bbb Z}[-])$ and the negative $K$-groups
$K_{-i}({\Bbb Z}[-])$, for $i\geq 1$, 
vanish for any subgroup of the groups 
considered in Theorem \ref{fic-braid}.\end{cor}

The paper is organized as follows. In Section 2 
we recall some background on Artin groups and in Section 3
we state our main results. 
We describe orbifold braid groups, some work from \cite{All}
and also prove a crucial proposition in Section 4,   
which are the key inputs in this work. Section 5 contains the proofs
of Theorems \ref{orbi-poly-free} and \ref{d-b}.
The proof of Theorem \ref{surgery}, on the 
computation of the surgery groups of pure Artin groups, is   
given in Section 6.

\begin{ack}{\rm The author is grateful to the referees for their
    valuable comments and suggestions, which made some significant
    improvements in the paper.}\end{ack}

\section{Artin groups}
The Artin groups are an important class of groups, and appear  
in different areas of Mathematics.

We are interested in those Artin 
groups, which appear as 
extensions of Coxeter groups by the fundamental group  
of hyperplane arrangement complements in ${\Bbb C}^n$, for some $n$. We need 
the topology and geometry of this complement to work on the Artin groups.

The Coxeter groups are generalization of reflection 
groups, and is yet another useful class of groups (\cite{Cox}). Several Coxeter groups 
appear as Weyl groups of simple Lie algebras. In fact, all the Weyl groups are 
Coxeter groups. There are two different types of reflection groups we consider 
in this article; spherical and affine. The spherical type is generated by reflections 
on hyperplanes passing through the origin in an Euclidean space, and the affine type 
is generated by reflections on hyperplanes without the requirement that 
the hyperplanes pass through the origin.

Next we give the presentation of the Coxeter groups 
in terms of generators and relations. This way one has a better understanding 
of the corresponding Artin groups.

Let $K=\{s_1,s_2,\ldots,s_k\}$ be a finite set, and $m:K\times K\to
\{1,2,\cdots, \infty\}$ be a 
function with the property that $m(s,s)=1$, and $m(s',s)=m(s,s')\geq 2$ for $s\neq s'$. 
The {\it Coxeter group} associated to the pair $(K,m)$, is by definition the 
following group.
$${\cal W}_{(K,m)}=\langle K\ |\ (ss')^{m(s,s')}=1,\ s,s'\in S\ \text{and}\ m(s,s')<\infty\rangle.$$
$(K,m)$ is called a {\it Coxeter system}.

Associated to a Coxeter system there is a graph called {\it Coxeter diagram}. This 
graph has vertex set $K$, two vertices $s$ and $s'$ are connected by an edge if $m(s,s')\geq 3$. 
If $m(s,s')=3$, then the edge is not labeled, otherwise it gets the label $m(s.s')$.
From this diagram, the presentation of the Coxeter group can be reproduced.

A complete 
classification of finite, irreducible Coxeter groups is known (see
\cite{Cox}). Here  
irreducible means the corresponding Coxeter diagram is connected. That is, the Coxeter group becomes 
direct product of the Coxeter groups corresponding to the connected components of the 
Coxeter diagram. In this article, 
without any loss, by a Coxeter group we will always mean an
irreducible Coxeter group. Since the Artin group (see below for
definition) corresponding to a
reducible Coxeter group 
is the direct product of the Artin groups of its
irreducible factors (see Condition $(3)$ in Definition \ref{C}).
We reproduce the list of all finite Coxeter groups in Table 1. These 
are exactly the finite (spherical) reflection groups. For a general reference on  
this subject we refer the reader to \cite{Hum}.

Also there are infinite Coxeter 
groups which are affine reflection groups. 
See Table 2, which shows a list 
of only those, which we need for this paper. For a complete list see 
\cite{Hum}. In the tables the associated 
Coxeter diagrams are also shown.

The symmetric groups $S_n$ on $n$ letters and the finite dihedral group 
are examples of Coxeter groups. 
These are the Coxeter group of type $A_n$ and $I_2(p)$ respectively in the table. 
The {\it Artin group} associated to the Coxeter
group ${\cal W}_{(K, m)}$ is, by definition, 
$${\cal A}_{(K, m)}=\l K\ |\ ss'ss'\cdots = s'ss's\cdots;\ s,s'\in K\r.$$
Here the number of times the factors in $ss'ss'\cdots$ appear is 
$m(s,s')$; for example, if $m(s,s')=3$, then the relation is $ss's=s'ss'$. 
${\cal A}_{(K,m)}$ is called the Artin group of type 
${\cal W}_{(K,  m)}$. There is an obvious surjective homomorphism 
${\cal A}_{(K, m)}\to {\cal W}_{(K, m)}$. The kernel ${\cal {PA}}_{(K, m)}$ (say) of this 
homomorphism is called the associated {\it pure Artin group}. 
In the case of type $A_n$, the (pure) Artin 
group is also known as the {\it classical (pure) braid group}. We will justify the word 
`braid' in Section 4. When a  
Coxeter group is a finite or an affine reflection group, the associated Artin group is called 
of {\it finite} or {\it affine type}, respectively. 

\medskip
\centerline{\includegraphics[height=9.5cm,width=7.5cm,keepaspectratio]{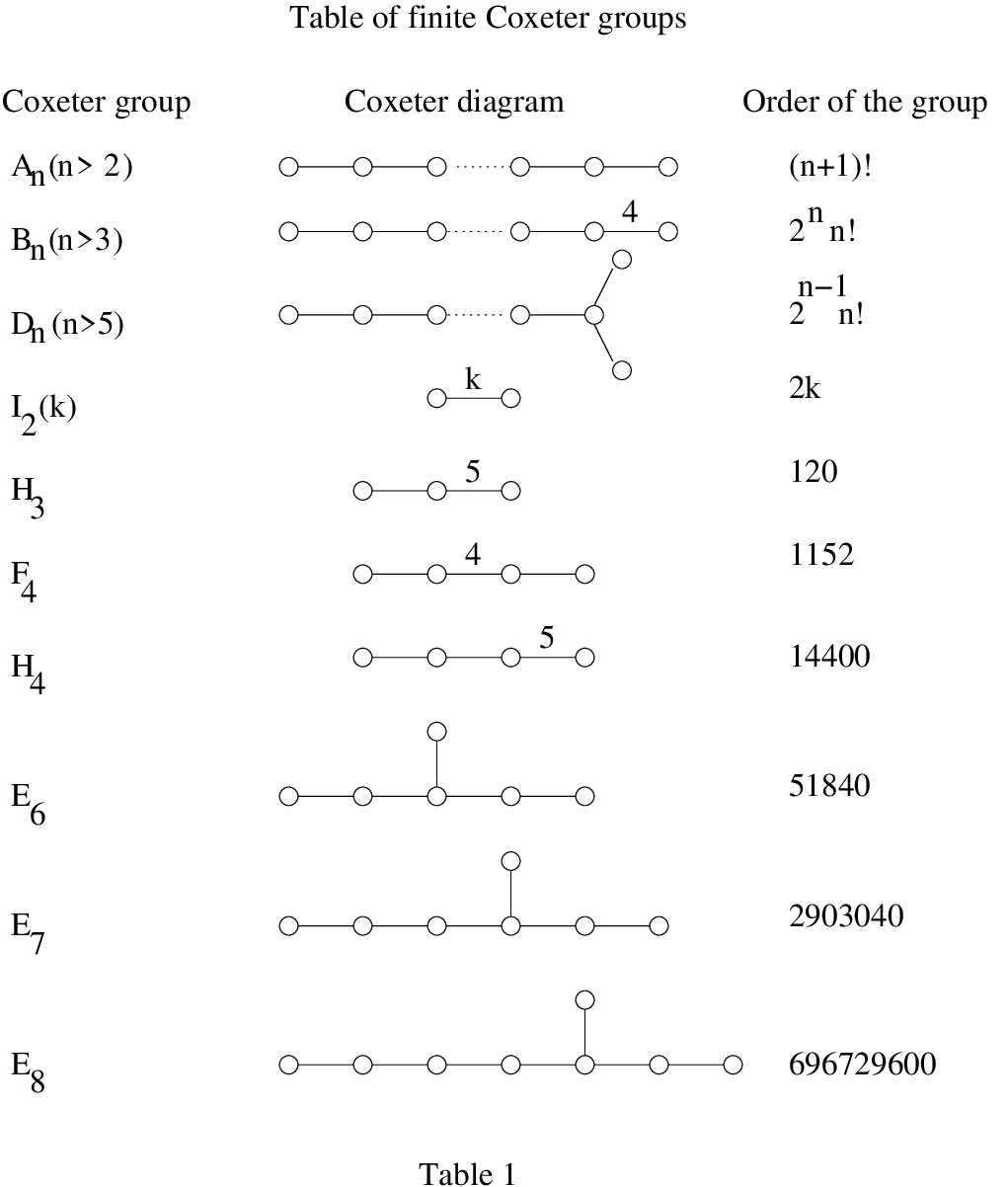}}

\medskip
\centerline{\includegraphics[height=8.5cm,width=6.5cm,keepaspectratio]{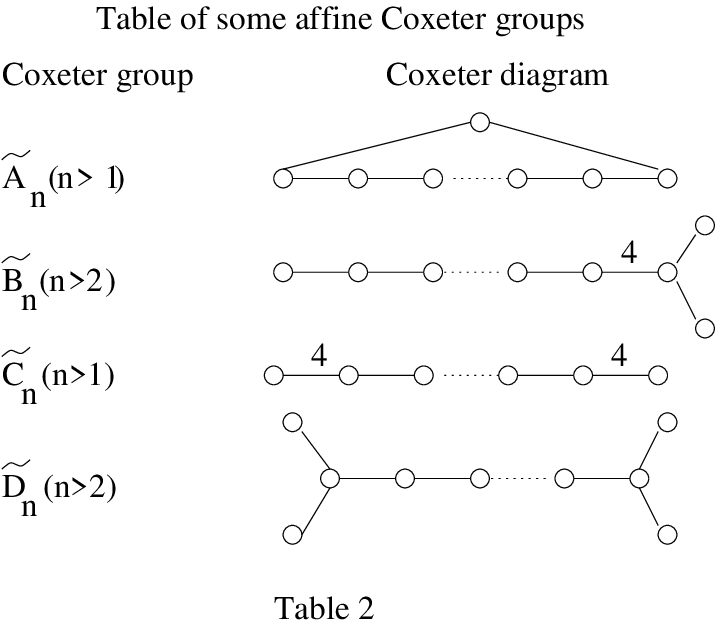}}

\medskip

Now given a finite Coxeter group (equivalently, a finite reflection
group) ${\cal W}_{(K, m)}$, and its standard faithful 
representation in $GL(n, {\Bbb R})$, for some $n$, consider the 
hyperplane arrangement in ${\Bbb R}^n$, where each hyperplane is fixed  
pointwise by an involution of the group. Next complexify ${\Bbb R}^n$ to ${\Bbb C}^n$ and 
consider the corresponding complexified hyperplanes. We call these 
complex hyperplanes in this arrangement, {\it reflecting 
hyperplanes} associated to the reflection group ${\cal W}_{(K, m)}$. Let $PA_{(K, m)}$ be the 
complement of this hyperplane arrangement in ${\Bbb C}^n$. The fundamental 
group of $PA_{(K,m)}$ is identified with the 
pure Artin group ${\cal {PA}}_{(K, m)}$, associated to the reflection group (\cite{Bri1}). 
The group ${\cal W}_{(K, m)}$ acts 
freely on this complement. The quotient space has 
fundamental group isomorphic to the Artin group, associated to 
the reflection group (\cite{Bri1}). That is, we get an 
exact sequence of the following type.

\centerline{
\xymatrix{1\ar[r]&{\cal {PA}}_{(K, m)}\ar[r]&{\cal A}_{(K, m)}\ar[r]&{\cal W}_{(K, m)}\ar[r]&1.}}


We call the space ($PA_{(K, m)}$) $PA_{(K, m)}/{\cal W}_{(K, m)}$ the {\it (pure) Artin space} of the 
Artin group of type ${\cal W}_{(K, m)}$. 

In this article, this geometric interpretation 
of the Artin groups is relevant for us. 

Next our aim is to 
discuss some generalization of this interpretation, replacing ${\Bbb C}$ by some 
$2$-dimensional orbifold, in the classical braid group case.
These generalized groups are related to some of the Artin groups. This connection  
is exploited here to help extend our earlier 
result in \cite{FR} and \cite{R5}, to the Artin groups of types  
$A_n$, $B_n (=C_n)$, $D_n$,  $\tilde A_n$,  
$\tilde B_n$ and $\tilde C_n$. Then using a different method 
we treat the cases $F_4$, $G_2$ and $I_2(p)$. Although, we will see the
core idea behind the proofs is the same; some kind of inductive
structure, 
with $3$-manifold groups as building blocks.

\begin{rem}\label{complex}{\rm Recall that, a complex reflection group is a 
subgroup of $GL(n, {\Bbb C})$ generated by complex reflections. 
A {\it complex reflection} is an element of $GL(n, {\Bbb C})$ which fixes a 
hyperplane in ${\Bbb C}^n$ pointwise. There is a classification of finite 
complex reflection groups (\cite{ST}) in the following two classes.
 
1. $G(de,e,r)$, where $d,e,r$ are positive integers. 

2. 34 exceptional groups 
denoted by $G_4,\ldots G_{37}$. 

There are associated Artin groups 
of the finite complex reflection groups.}\end{rem}

\begin{defn}\label{braid-space} {\rm For a connected topological space $X$, we define 
the {\it pure braid space} of $X$ of $n$ strings by the following.
$$PB_n(X):=X^n-\{(x_1,x_2,\ldots , x_n)\in X^n\ |\ x_i=x_j\
\text{for some}\ i\neq j\}.$$

For $n=1$ we define $PB_1(X)=X$. For $n\geq 2$, 
the symmetric group $S_n$ acts freely on the pure braid space, by
permuting coordinates. The 
quotient space $PB_n(X)/S_n$ is denoted by $B_n(X)$, and is called the {\it braid space} 
of $X$. The fundamental group of ($PB_n(X)$) $B_n(X)$ is called the
({\it pure}) {\it braid group} of $n$ strings of the topological space $X$.}\end{defn}

We will need the following well-known Fadell-Neuwirth fibration theorem.

\begin{thm}(\cite{FN}) \label{fn} Let $M$ be a connected manifold of
  dimension $\geq 2$. Then  
  the projection $M^n\to M^{n-1}$ to the first $n-1$ coordinates induces a locally trivial
  fibration $PB_n(M)\to PB_{n-1}(M)$, with fiber homeomorphic to $M- \{(n-1)\ \text{points}\}$.\end{thm}

\section{Statements of results}
In this section we state the results we referred to in the Introduction.

In \cite{FR} and \cite{R5}, we  
proved the fibered isomorphism conjecture for $H\wr F$, where $H$ is any subgroup of 
the classical braid group and $F$ is a finite group.

In this article we consider a general 
statement of the conjecture with coefficients, and we denote the {\it isomorphism 
conjecture with coefficients and finite wreath product in the $K$-,
$L$- and $A$-theories} by $FICwF$.
The finite wreath product version of the conjecture was 
introduced in \cite{Rou0}, and its general properties were proved in \cite{R1}. 
The importance of this version of the conjecture was first 
observed in \cite{FR} and \cite{Rou0}.
We do not need the exact statement of the
isomorphism conjecture, but need some already know results. Therefore, we do not state  
the conjecture and refer the reader to \cite{R1} or \cite{lueck} for more on this
subject.

Now we are in a position to state the results we prove in this paper. 

We need to define the following class of groups, which 
contains the groups defined in [\cite{Rou01}, Definition 1.2.1]. 

\begin{defn}\label{C}{\rm Let ${\cal C}$ denote the smallest class of groups 
satisfying the following conditions.

1. The fundamental group of any connected manifold of dimension $\leq 3$ belongs 
to ${\cal C}$.

2. If $H$ is a subgroup of a group $G$, then $G\in {\cal C}$ 
implies $H\in {\cal C}$. This reverse implication is also true if 
$H$ is of finite index in $G$.

3. If $G_1, G_2\in {\cal C}$ then $G_1\times G_2\in {\cal C}$.

4. If $\{G_i\}_{i\in I}$ is a directed system of groups and $G_i\in {\cal C}$ for each 
$i\in I$, then the $colim_{i\in I}G_i\in {\cal C}$. 

5. Let $1\to K\to G\to H\to 1$ be a short exact sequence of groups with 
$p:G\to H$ being the last surjective homomorphism. If $K$, $H$ and 
$p^{-1}(C)$, for any infinite cyclic subgroup $C$ of $H$, belong to 
$\cal C$ then $G$ also belongs to $\cal C$.}\end{defn}

In the following two theorems, we prove that the Artin groups considered in
this paper belong to $\cal C$.

\begin{thm}\label{orbi-poly-free} Let $S$ be an aspherical connected $2$-manifold. 
  Then the (pure) braid group of $S$ belongs to $\cal C$. Consequently, the Artin groups
  of types $A_n$, $B_n (=C_n)$, $\tilde A_n$ and $\tilde C_n$ belong to
$\cal C$.\end{thm}

\begin{thm}\label{d-b} The Artin groups of types $D_n$, $F_4$, $G_2$, $I_2(p)$,
  $\tilde B_n$ and $G(de,e,r)$, $d,r\geq 2$  belong to
  $\cal C$.\end{thm}

Theorem \ref{fic-braid} is then a consequence of Theorems \ref{orbi-poly-free}, \ref{d-b} and the 
following theorem.

\begin{thm} \label{ficwf} The FICwF is true for any group in $\cal C$.\end{thm}

\begin{proof} This is a special case of [\cite{lueck}, Theorem 20.8.6].\end{proof}

\begin{rem}{\rm It will be interesting to know if all the Artin groups belong to
    $\cal C$.}\end{rem}

  Finally, we state a corollary regarding an explicit computation 
of the surgery groups $L_*({\Bbb Z}[-])$ of the pure Artin groups of finite type. This comes out 
of the isomorphisms of the two assembly maps $A_K$ and $A_L$, for the
finite type Artin groups, stated in 
Theorem \ref{fic-braid}. The calculation was done 
in \cite{R5} for the classical pure braid group case. Together with the 
isomorphisms of $A_K$ and $A_L$, the proof 
basically used the homotopy type of the suspension of the pure braid space and  
the knowledge of the surgery groups of the trivial group. 

\begin{thm}\label{surgery} The surgery groups of the finite type pure Artin groups 
take the following form.
$$
L_i({\cal {PA}})\ =\ \begin{cases}{\Bbb Z} & \text{if}\ i=4k,\\
{\Bbb Z}^N & \text{if}\ i=4k+1,\\
{\Bbb Z}_2 & \text{if}\ i=4k+2,\\
{\Bbb Z}_2^N & \text{if}\ i=4k+3.
\end{cases}$$
Here $N$ is the number of reflecting hyperplanes associated to the finite Coxeter group, 
as given in the following table.

\centerline{
\begin{tabular}{|l|l|}
\hline
{\bf ${\cal {PA}}$\ =\ pure Artin group of type} & $N$\\
\hline \hline
$A_{n-1}$&$\frac{n(n-1)}{2}$\\
\hline
$B_n (=C_n)$ & $n^2$\\
\hline
$D_n$& $n(n-1)$\\ 
\hline
$F_4$& $24$\\
\hline
$I_2(p)$& $p$\\
\hline
$G_2$& $6$\\
\hline
\end{tabular}}
\end{thm}
\centerline{Table 3}

\begin{rem}{\rm 
Recall that, there are surgery groups for different kinds of 
surgery problems, and 
they appear in the literature with the notations 
$L^*_i(-)$, where $*=h,s, \l -\infty \r$ or 
$\l j\r$ for $j\leq 0$. But all of them are naturally 
isomorphic for a group  
$G$, if $Wh(G)$, $\tilde K_0({\Bbb Z}[G])$, and $K_{-i}({\Bbb Z}[G])$, for $i\geq 1$, 
vanish. This can be checked by the Rothenberg exact sequence ([\cite{Sha}, 4.13]).
$$\cdots\to L^{\l j+1\r}_i(R)\to L^{\l j\r}_i(R)\to \hat{H}^i({\Bbb Z}/2; 
\tilde K_j(R))\to L^{\l j+1\r}_{i-1}(R)\to L^{\l j\r}_{i-1}(R)\to\cdots .$$
Where $R={\Bbb Z}[G]$, $j\leq 1$, $Wh(G)=\tilde K_1(R)$, $L_i^{\l1\r}=L_i^h$, 
$L_i^{\l2\r}=L_i^s$ and 
$L_i^{\l-\infty\r}$ is the limit of $L_i^{\l j\r}$.
Therefore, because of Corollary \ref{lkt}, we use the  
simplified notation $L_i(-)$ in the above corollary.}\end{rem}

\begin{rem}\label{torsion-free}{\rm The same calculation holds for the other pure Artin groups 
corresponding to the finite type Coxeter groups and finite complex 
reflection groups (see Remark \ref{complex}) also, provided we know 
the $FICwF$ for these groups . We further need the fact that the 
Artin spaces are aspherical (\cite{Del}), which implies that the  
Artin groups are torsion free. The last fact can also be proved group 
theoretically (\cite{Gar}). The Artin groups corresponding to 
the finite complex reflection groups $G(de,e,r)$ is  
torsion free, since  
${\cal A}_{G(de,e,r)}$ is a subgroup of ${\cal A}_{B_r}$, for $d,r\geq
2$ ([\cite{CLL},
Proposition 4.1]).}\end{rem}

\section{Artin groups and orbifold braid groups}
In this section we give a short introduction to orbifold braid groups and
describe a connection with Artin groups.

Let ${\Bbb C}(k,l; q_1,q_2,\cdots ,q_l)$ denote the complex plane with $k$ punctures at
the points $p_1,p_2,\cdots, p_k\in {\Bbb C}$, $l$
distinguished points (also called {\it cone} points)
$x_1,x_2,\cdots, x_l\in {\Bbb C}-\{p_1,p_2,\cdots, p_k\}$ and an integer $q_i>1$ attached
to $x_i$ for $i=1,2,\cdots, l$. $q_i$ is called the {\it order} of the cone point $x_i$.

One can define (pure) braid group of $n$ strings of ${\Bbb C}(k,l; q_1,q_2,\cdots ,q_l)$. There are two different
ways one can do this, which give the same result. The first
one is topological and the second is pictorial way as in the classical
braid group case. In the former, one realizes ${\Bbb C}(k,l; q_1,q_2,\cdots ,q_l)$ as a $2$-dimensional orbifold and
then considers $PB_n({\Bbb C}(k,l; q_1,q_2,\cdots ,q_l))$. Note that, $PB_n({\Bbb C}(k,l; q_1,q_2,\cdots ,q_l))$ is a high dimensional
orbifold and one considers its orbifold fundamental group and follow Definition \ref{braid-space}. We call this
orbifold fundamental group as the {\it (pure) orbifold braid group} of $n$
strings of ${\Bbb C}(k,l; q_1,q_2,\cdots ,q_l)$. For basics on orbifolds see \cite{Thu}.
The later pictorial definition is relevant for us. We describe it now from \cite{All}.

Recall that any element of the classical braid group $\pi_1(B_n({\Bbb C}))$, 
can be represented by a 
braid as in the following picture. Juxtaposing one braid over another gives 
the group operation.

\medskip

\centerline{\includegraphics[height=6cm,width=8cm,keepaspectratio]{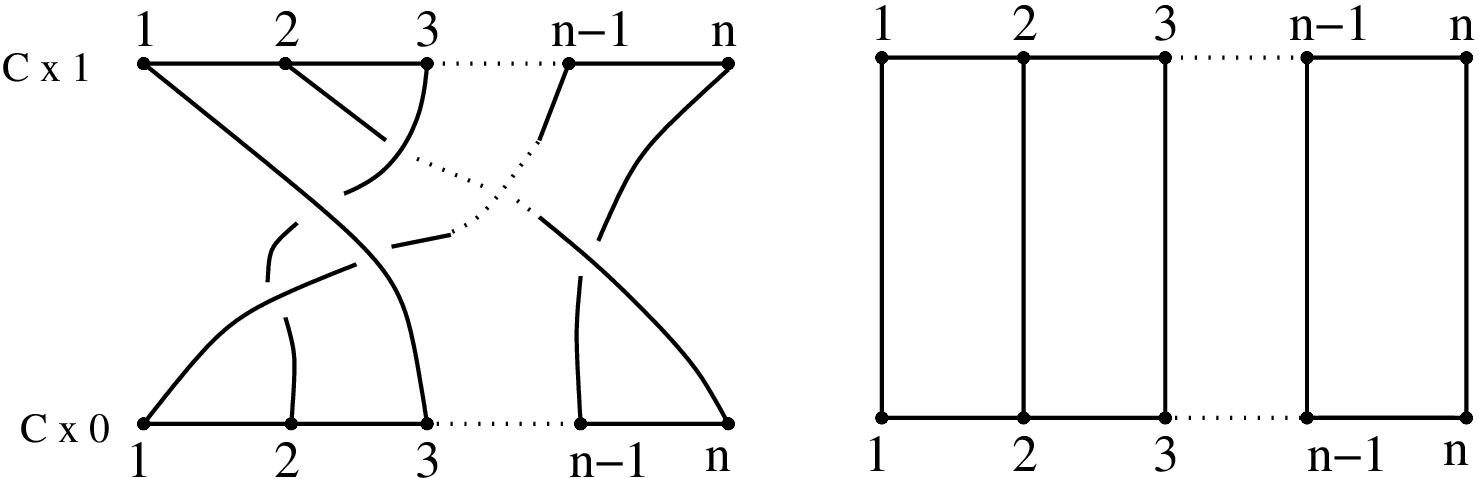}}

\centerline{Figure 1: A typical braid and the trivial (identity) braid.}

\medskip

Let $S$ be the complex plane with only one cone 
point. The underlying topological space of $S$ is nothing but the complex plane. 
Therefore, $B_n(S)$ is an orbifold when we consider the orbifold structure 
of $S$, otherwise it is the classical braid space. 

Therefore, although the fundamental group 
of the underlying topological space of $B_n(S)$ has the classical 
pictorial braid representation as above, the orbifold fundamental group 
of $B_n(S)$ needs different treatment.
We point out here a similar pictorial braid representation 
of a typical element of the orbifold fundamental
group of $B_n(S)$ from \cite{All}. The following picture 
shows the case for one cone point $x$. The thick line 
represents $x\times I$.

\medskip

\centerline{\includegraphics[height=4cm,width=7cm,keepaspectratio]{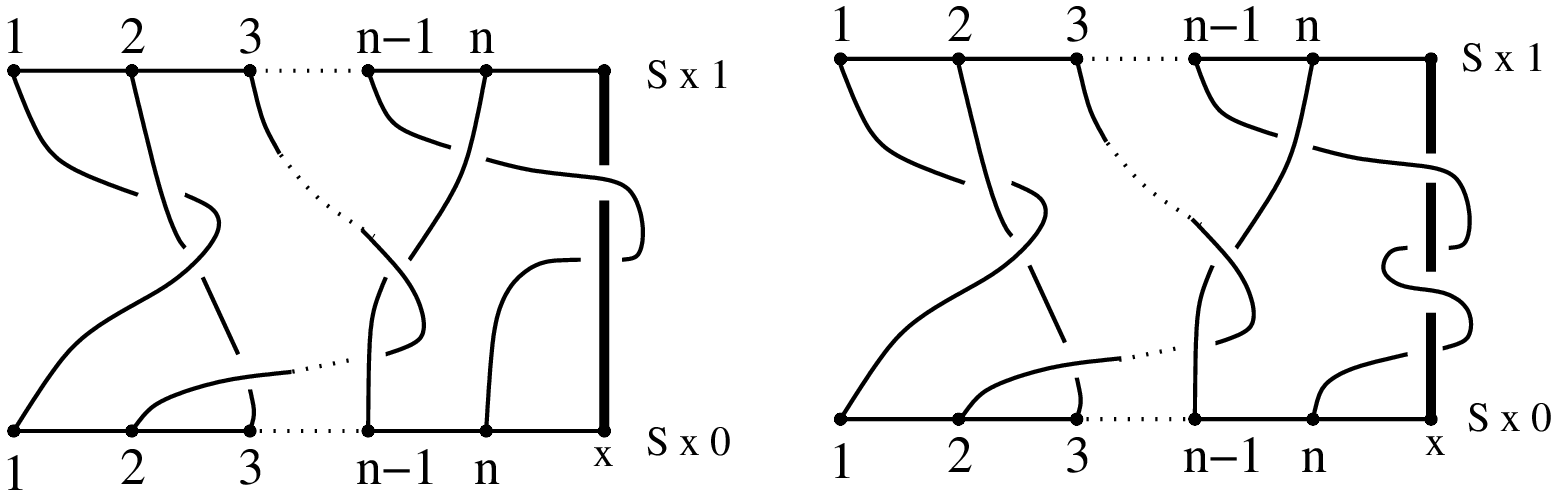}}

\centerline{Figure 2: A typical orbifold braid.}

\medskip

The group operation is again given by juxtaposing one braid onto another, and the
identity element is also obvious.
 
Note that, in Figure 2, both the braids represent the same element in $\pi_1(B_n(S))$, but 
different in $\pi_1^{orb}(B_n(S))$, depending on the order of the cone point.

In this 
orbifold situation the movement of the braids is restricted, because 
of the presence of the cone point. Therefore,  
one has to define new relations among braids, 
respecting the singular set of the orbifold. We produce one situation 
to see how this is done. 
The second picture in Figure 3 represents part of a typical element of $\pi_1^{orb}(B_n(S))$.

\medskip

\centerline{\includegraphics[height=4cm,width=7cm,keepaspectratio]{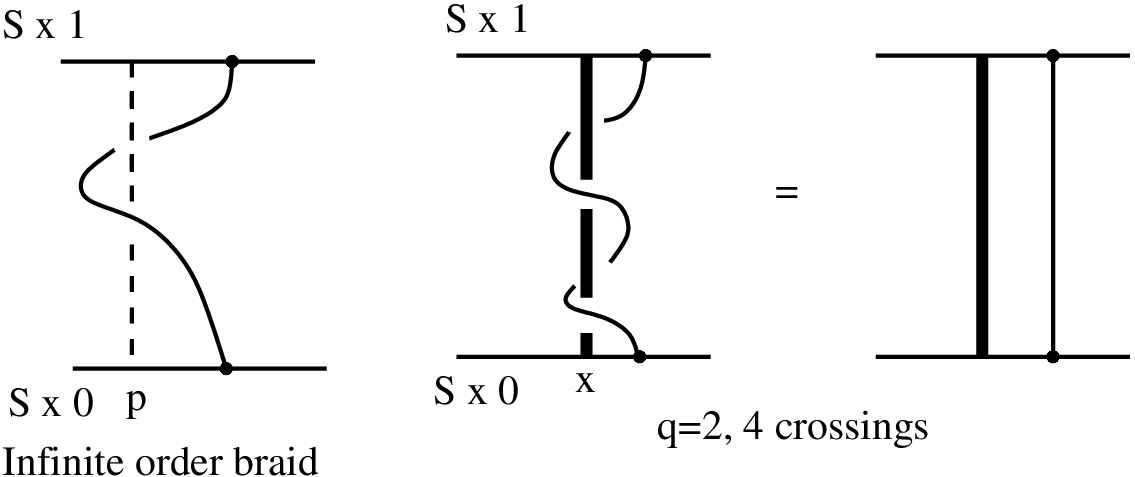}}

\centerline{Figure 3: Braid movement around a puncture and a cone point.}

\medskip

Now if a string in the braid wraps the thick line $x\times I$, $q$
times, then   
it is equal to the third picture. This is because, if a loop circles 
$q$ times around the cone point $x$, then the loop gives the 
trivial element in the orbifold fundamental group of $S$. Therefore, both braids represents 
the same element in the orbifold fundamental group of $B_n(S)$. Furthermore, if the 
string wraps the thick line, but less than $q$ number of times, then it is 
not equal to the unwrapped braid. For more details 
see \cite{All}.

When there is a puncture $p$ in the complex plane then the
braids will have to satisfy a similar property, but in this case, if
any of the string wraps $p\times I$ once then the braid will have infinite order. See
the first picture in Figure 3.

There is a useful connection between 
the orbifold braid group of the orbifold ${\Bbb C}(k,l;q_1,q_2,\cdots, q_l)$ and some of 
the Artin groups. More precisely, to associate braid type 
representation to elements of an Artin group. 

For a given Artin group one has to choose a 
suitable $2$-dimensional orbifold, as in Table 4, in
Theorem \ref{AllC}. In Remark \ref{complex}, we 
recalled a classification of finite complex reflection groups. In the
case of finite complex reflection groups of type $G(de,e,r)$, for $d,r\geq
2$, one has to consider 
the braid group of the punctured plane (annulus). See the case $B_n$ 
in Table 4 in Theorem
\ref{AllC}, and use [\cite{CLL}, Proposition 4.1].

\begin{thm}(\cite{All}) \label{AllC} Let $A$ be an Artin group, and $S$ be an orbifold as 
described in the following table. Then $A$ is a (normal) 
subgroup of the orbifold braid group $\pi_1^{orb}(B_n(S))$. The 
third column gives the quotient group $\pi_1^{orb}(B_n(S))/A$.

\medskip
\centerline{
\begin{tabular}{|l|l|l|l|}
\hline
A={\bf Artin group of type} & {\bf Orbifold $S$}&{\bf Quotient group}&$n$\\
\hline \hline
$A_{n-1}$&${\Bbb C}$&$<1>$&$n>1$\\
\hline
$B_n(=C_n)$ & ${\Bbb C}(1,0)$ &$<1>$&$n>1$\\
\hline
$D_n$&${\Bbb C}(0,1;2)$ &${\Bbb Z}/2$& $n>1$\\ 
\hline
&&&\\
$\tilde A_{n-1}$& ${\Bbb C}(1,0)$ &${\Bbb Z}$& $n>2$\\
\hline
&&&\\
$\tilde B_n$&${\Bbb C}(1,1;2)$  &${\Bbb Z}/2$& $n>2$\\
\hline
&&&\\
$\tilde C_n$&${\Bbb C}(2,0)$ &$<1>$&$n>1$\\
  \hline
  &&&\\
  $\tilde D_n$&${\Bbb C}(0,2;2,2)$&${\Bbb Z}_2\times{\Bbb Z}_2$&$n>2$\\
  \hline
\end{tabular}}
\medskip
\centerline{\rm{Table 4}}
\end{thm}

\begin{proof}
  See \cite{All} for the proof.\end{proof}

The following proposition is crucial for this paper, which helps deducing some more
inclusions as in the previous theorem.

\begin{prop} \label{mapping-braid} There is an injective homomorphism from the orbifold braid group 
  $\pi_1^{orb}(B_n({\Bbb C}(1,1;q)))$
  to $\pi_1^{orb}(B_{n+1}({\Bbb C}(0,1;q)))$. That is, the orbifold braid group
  of $n$ strings of ${\Bbb C}(1,1;q)$ can be embedded into the
  orbifold braid group of $n+1$
  strings of ${\Bbb C}(0,1;q)$.\end{prop}

\begin{proof} We give a pictorial proof (following \cite{All}) and also define the explicit map.
  
  Let $B\in \pi_1^{orb}(B_n({\Bbb C}(1,1;q)))$ as in the following figure. Here $x$ is the cone
  point and $p$ is the puncture. We send $B$ to $\overline{B}$ where the line $p\times I$ is
  sent to a new string ($n+1$ as in the figure) and hence $\overline{B}\in\pi_1^{orb}(B_{n+1}({\Bbb C}(0,1;q)))$.
  
  It is easy to see that the map $B\mapsto \overline{B}$ is an
  injective homomorphism. Here recall that the composition is
  juxtaposition of braids. In the first picture of Figure 4, $S={\Bbb C}(1,1;q)$
  and it is ${\Bbb C}(0,1;q)$ in the second one.

\medskip

\centerline{\includegraphics[height=9cm,width=11cm,keepaspectratio]{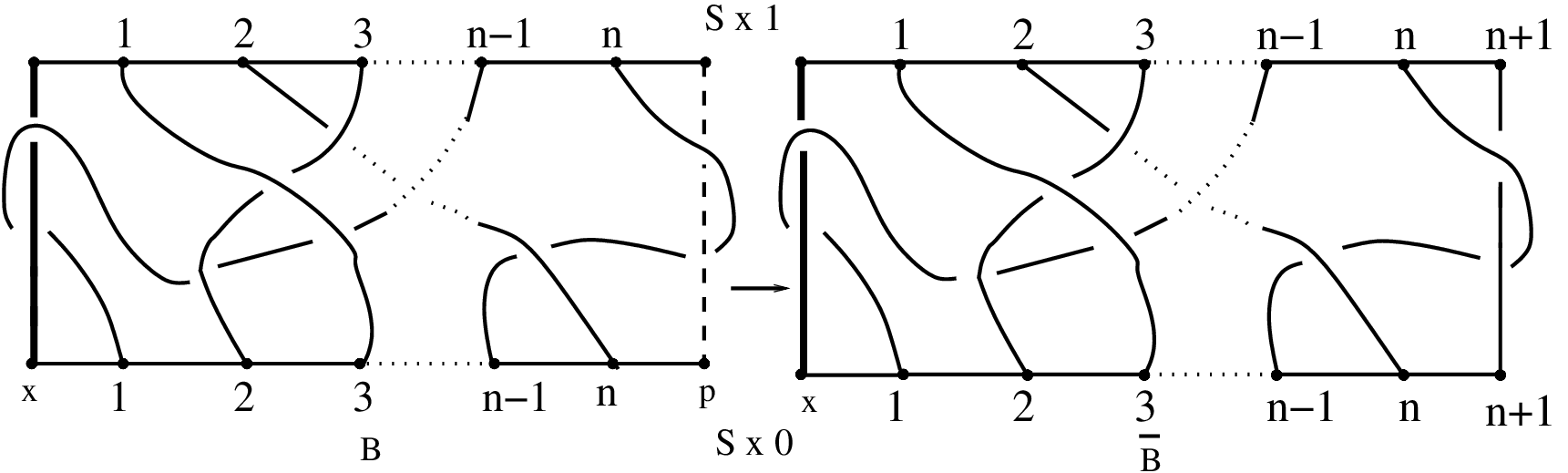}}

\centerline{Figure 4: Mapping a braid after filling a puncture with a string.}

\medskip

The group $\pi_1^{orb}(B_n({\Bbb C}(1,1;q)))$ has the following generators,  
$$X,A_1,\cdots, A_{n-1}, P.$$

\medskip

\centerline{\includegraphics[height=11cm,width=11cm,keepaspectratio]{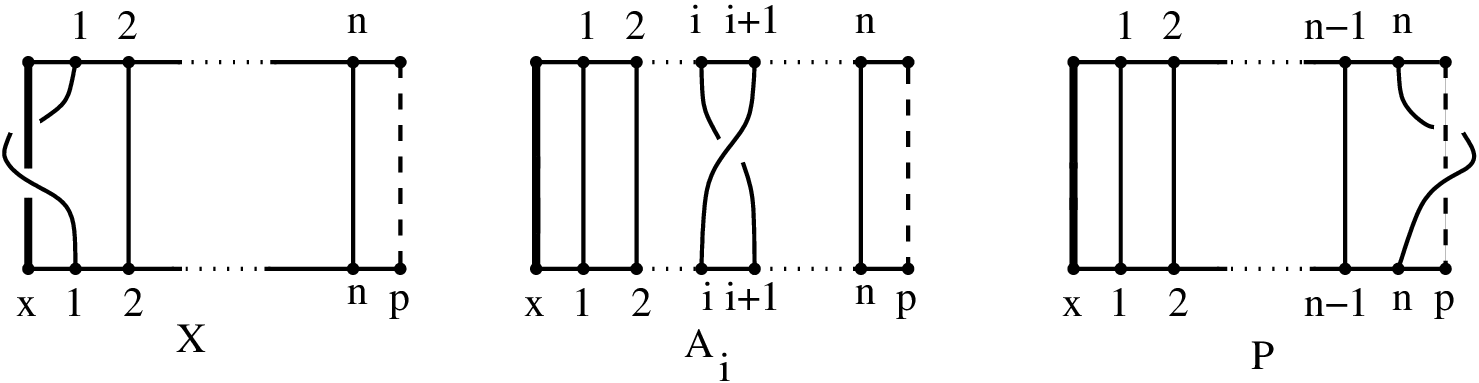}}

\centerline{Figure 5: Generators of $\pi_1^{orb}(B_n({\Bbb C}(1,1;q)))$.}

\medskip

The relations are $$X^q=1, XA_1XA_1=A_1XA_1X, A_iA_{i+1}A_i=A_{i+1}A_iA_{i+1},$$$$
  A_iA_j=A_jA_i,\ \text{for}\ |i-j|>1\ \text{and}\ PA_{n-1}PA_{n-1}=A_{n-1}PA_{n-1}P.$$

And the group $\pi_1^{orb}(B_{n+1}({\Bbb C}(0,1;q)))$ has the following
generators.

$$\overline X,\overline A_1,\cdots, \overline A_{n-1},  \overline A_n.$$
\medskip

\centerline{\includegraphics[height=11cm,width=11cm,keepaspectratio]{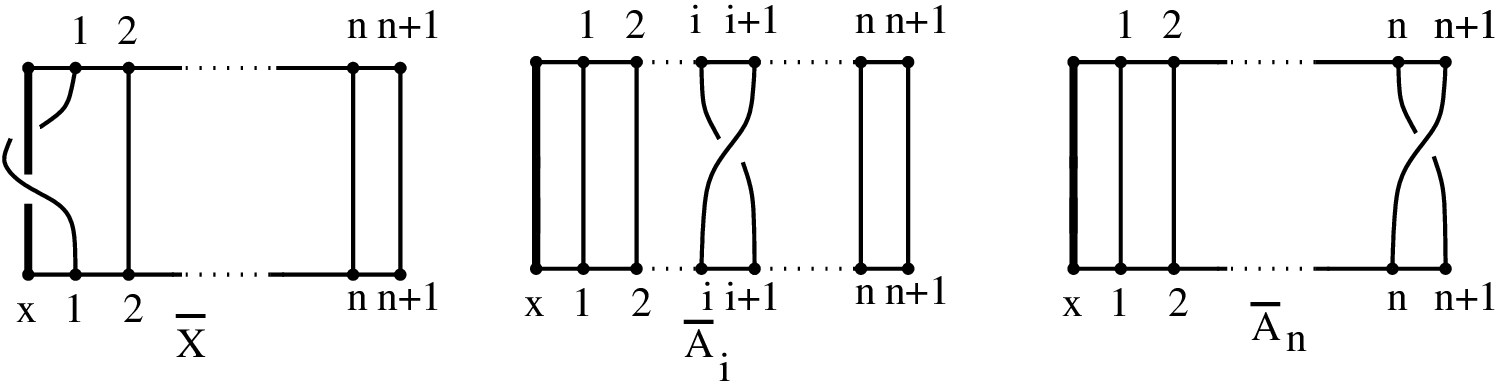}}

\centerline{Figure 6: Generators of $\pi_1^{orb}(B_{n+1}({\Bbb C}(0,1;q)))$.}

\medskip

The relations are 
$$\overline X^q=1,\ \overline X\ \overline A_1\overline X\ \overline A_1=
\overline A_1\overline X\ \overline A_1\overline X,\ \overline A_i\overline A_{i+1}\overline A_i
=\overline A_{i+1}\overline A_i\overline A_{i+1}$$
$$\ \text{and}\ \overline A_i\overline A_j=\overline A_j\overline A_i,\ \text{for}\ |i-j|>1.$$

See \cite{All} for more on this matter.

It is easy to see that the map defined above is obtained by the following
map on the generators level.
$$X\mapsto \overline X, A_i\mapsto \overline A_i\ \text{for}\
i=1,2,\ldots, n-1\ \text{and}\ P\mapsto \overline A_n^2.$$
This completes the proof of the proposition.
\end{proof}

\section{Proofs of Theorems \ref{orbi-poly-free} and \ref{d-b}}
In this section we give the proofs of Theorems \ref{orbi-poly-free} and \ref{d-b}.
Throughout this section we refer to the conditions in Definition \ref{C} as Condition $(n)$.
We need the following lemma for some induction argument.

\begin{lemma}\label{induction}
Let $f:M\to N$ be a locally trivial fibration of connected 
manifolds with fiber a $2$-manifold, and assume that $\pi_2(N)=1$. If $\pi_1(N)\in {\cal C}$, 
then $\pi_1(M)\in {\cal C}$.\end{lemma}
\begin{proof} Let $F$ be the $2$-manifold fiber of the fibration over
  some base point.
  Then we have the following short exact sequence of groups
  coming from the long exact sequence of homotopy groups applied to
  the fibration $f$.

  \medskip
  \centerline{
    \xymatrix{1\ar[r]&\pi_1(F)\ar[r]&\pi_1(M)\ar[r]^{f_*}& \pi_1(N)\ar[r]& 1.}}
  \medskip
  
  By Conditions $(1)$ and $(2)$, $\pi_1(M)\in {\cal C}$ if $\pi_1(N)$ is finite. Therefore,
  we can assume that $\pi_1(N)$ is infinite. Let $C < \pi_1(N)$ be an infinite cyclic
  subgroup generated by $[\alpha]\in\pi_1(N)$. Let $f_{\alpha}$ be the monodromy homeomorphism of $F$, corresponding
  to $\alpha$. Then we get the following first diagram, as the mapping torus $M_{f_{\alpha}}$ of $f_{\alpha}$ is identified
  with the pullback $\{(s,m)\in {\Bbb S}^1\times M\ |\ \alpha(s)=f(m)\}$ of the fibration $f$, under $\alpha:{\Bbb S}^1\to N$.
  The second diagram is obtained by applying the long exact homotopy sequence and its naturality
  on the two fibrations $f$ and $M_{f_\alpha}\to {\Bbb S}^1$.

Note that $\pi_1(M_{f_{\alpha}})$ goes into $f_*^{-1}(C)$, which gives the third diagram.
    
\medskip
\centerline{
\xymatrix{
  &&1\ar[d]&1\ar[d]&1\ar[d]&1\ar[d]\\
  F\ar[r]^{id}\ar[d]&F\ar[d]&\pi_1(F)\ar[r]^{id}\ar[d]&\pi_1(F)\ar[d]&\pi_1(F)\ar[r]^{id}\ar[d]&\pi_1(F)\ar[d]\\
  M_{f_{\alpha}}\ar[r]\ar[d]&M\ar[d]^f&\pi_1(M_{f_{\alpha}})\ar[r]\ar[d]& \pi_1(M)\ar[d]^{f_*}&\pi_1(M_{f_{\alpha}})\ar[r]\ar[d]& f_*^{-1}(C)\ar[d]\\
  {\Bbb S}^1 \ar[r]^{\alpha}&N&\pi_1({\Bbb S}^1)\ar[r]^{\alpha_*}\ar[d] & \pi_1(N)\ar[d]&\pi_1({\Bbb S}^1)\ar[r]^{\alpha_*}\ar[d] & C\ar[d]\\
  &&1&1&1&1}}

\medskip
From the third commutative diagram, and applying the five lemma we get that $f_*^{-1}(C)$ is isomorphic to the
fundamental group of $M_{f_{\alpha}}$, which is a $3$-manifold. Hence, $f_*^{-1}(C)$ belongs
to $\cal C$ by Condition $(2)$. Using Condition $(5)$ we
  complete the proof of the lemma.
  \end{proof}

\begin{proof}[Proof of Theorem \ref{orbi-poly-free}] Recall that, $S$
  is an aspherical, connected $2$-manifold, and we have to show that
  the braid group of $S$ belong to $\cal C$.
  
  First, note that the second half of the theorem follows from Theorem
  \ref{AllC} and Condition $(2)$,
  once we prove that the braid group of $S$ belongs to $\cal C$.

  By Condition $(2)$, it is enough to prove that the
  pure braid group of $S$ belongs to $\cal C$.

  So, let $\G_n=\pi_1(PB_n(S))$ be the pure braid group of $S$. The proof
  is by induction on $n$.

\medskip
\noindent
{\bf Case $n=1$.} In this case $\G_1\simeq \pi_1(S)$, and hence the
proof is completed using Condition $(1)$.

\medskip
\noindent
{\bf Case $n\geq 2$.} Assume that we have proved the theorem for $\G_{n-1}$.
Consider the following projection to the first $n-1$ coordinates;
$PB_n(S)\to PB_{n-1}(S)$. This is a locally trivial fibration 
with fiber $S- \{(n-1)\ \text{points}\}$ (Theorem \ref{fn}). By the
induction hypothesis, $\pi_1(PB_{n-1}(S))\in \cal C$.
Therefore, by Lemma \ref{induction}, $\pi_1(PB_n(S))\in \cal C$.

This completes the proof of the Theorem.
\end{proof}

\begin{rem}\label{fth}{\rm The same proof also applies to show that the fundamental group
    of any fiber-type hyperplane arrangement (see [\cite{OT},
    Definition 5.11]) complement in
    ${\Bbb C}^n$ belongs to $\cal C$.}\end{rem}

\begin{proof}[Proof of Theorem \ref{d-b}] Except for the $\tilde B_n$ case, by Condition $(2)$, it is
  enough to prove that the pure Artin groups of the other types belong to $\cal C$. We will prove the 
  $\tilde B_n$ case at the end.

  First, we deduce the proof of the theorem in the $F_4$ and $D_n$ cases.
  
  The idea is to see the pure Artin spaces of types $F_4$ and $D_n$
  as the total space of a locally trivial fibration   
over another aspherical manifold, whose fundamental group belong to $\cal C$, and 
the fiber is a $2$-manifold. Then Lemma \ref{induction} will be applicable.

For the rest of the section we refer to the discussion in the
proof of [\cite{Bri}, Proposition 2].

We begin with the following lemma.

\begin{lemma} \label{fibration} Let $Z_n=\{(z_1, z_2,\ldots , 
z_n)\in {\Bbb C}^n\ |\ z_i\neq 0\ \text{for all}\ i, z_i\neq z_j\ \text{for}\ i\neq j\}$. 
Then $Z_n$ is aspherical and $\pi_1(Z_n)$ belongs to $\cal C$ for all $n$.\end{lemma}

\begin{proof} The proof that $Z_n$ is aspherical follows easily 
by an induction argument, on taking successive projections 
and using the long exact sequence 
of homotopy groups of a fibration. (See Theorem \ref{fn} and note that $Z_n=PB_n({\Bbb C}^*)$).

The proof that $\pi_1(Z_n)$ belongs to $\cal C$, follows from the first
half of Theorem \ref{orbi-poly-free}.
\end{proof}

\noindent
{\bf $F_4$ case.}

The pure Artin group of type $F_4$ is isomorphic to the 
fundamental group of the following hyperplane arrangement complement 
in ${\Bbb C}^4$.

$PA_{F_4}= \{(y_1,y_2,y_3,y_4)\in{\Bbb C}^4\ |\ y_i\neq 0\ 
\text{for all}\ i,\ y_i\pm y_j\neq 0\ \text{for}\ i\neq j,$ 

\hspace{1.6cm}$ y_1\pm y_2\pm y_3\pm y_4\neq 0\}.$

The signs above appear in arbitrary combinations.
Next consider the map 
$PA_{F_4}\to Z_3$ defined by the following formula.
$$z_i=y_1y_2y_3y_4(y_4^2-y_i^2), i=1,2,3.$$

This map is a locally trivial fibration, with fiber a $2$-manifold.
Hence, by Lemmas \ref{induction} and \ref{fibration}, $\pi_1 ({PA}_{F_4})$ belongs
to $\cal C$.

\noindent
{\bf $D_n$-case.}

Again, the reference for 
this discussion is [\cite{Bri}, Proposition 2].

This proof goes in the same line, as the proof of the 
$F_4$ case. First, note that 
the pure Artin space of type $D_n$ has the following form.
$$PA_{D_n}=\{(y_1,\ldots ,y_n)\in{\Bbb C}^n\ |\ y_i\pm y_j\neq 0\ \text{for}\ i\neq j\}.$$ 
Next the map $PA_{D_n}\to Z_{n-1}$ defined by $z_i=y_n^2-y_i^2$,
$i=1,2,\cdots, n-1$, is a
locally trivial fibration with fiber a $2$-manifold. 
The proof now follows from Lemmas \ref{induction} and \ref{fibration}.

\medskip
\noindent
{\bf $I_2(p)$-case.}

First, note that this is a 
rank $2$ case. Therefore, the pure Artin group of type $I_2(p)$ belongs to
$\cal C$, by the following more general statement.

\begin{lemma} The fundamental group of any hyperplane arrangement complement 
in ${\Bbb C}^2$ belongs to $\cal C$.\end{lemma}

\begin{proof}
  Let ${\cal H}$ be the union of the hyperplanes in ${\Bbb
      C}^2$. Consider the unit sphere ${\Bbb S}^3\subset {\Bbb
      C}^2$. Then ${\Bbb C}^2-{\cal H}$ deformation retracts to
    ${\Bbb S}^3-({\Bbb S}^3\cap {\cal H})$, which is a
    $3$-manifold. Hence, by Condition 1,
    $\pi_1({\Bbb C}^2-{\cal H})\in {\cal C}$.
\end{proof}

\medskip
\noindent
{\bf $G_2$-case.}

Since $G_2=I_2(6)$, the proof in this case is completed using the
previous case.

\medskip
\noindent
{\bf $G(de,e,r)$-case.}

For the proof of the theorem in the $G(de,e,r)$-case, we refer to Remark \ref{complex}.
In ([\cite{CLL}, Proposition 4.1]) it was shown that, for $d,r\geq 2$,  
${\cal A}_{G(de,e,r)}$ is a subgroup of ${\cal A}_{B_r}$. 
(In \cite{CLL}, these are called the braid groups associated to the reflection groups). 
Therefore, 
using Condition $(2)$ and the $B_n$ case of Theorem \ref{orbi-poly-free}, ${\cal A}_{G(de,e,r)}\in \cal C$.

\medskip
\noindent
{\bf $\tilde B_n$-case.}

By Theorem \ref{AllC}, the Artin group of type $\tilde B_n$ is
isomorphic to a subgroup of $\pi_1^{orb}(B_n({\Bbb C}(1,1;2)))$, and by Proposition \ref{mapping-braid},
$\pi_1^{orb}(B_n({\Bbb C}(1,1;2)))$ is isomorphic to a subgroup of $\pi_1^{orb}(B_{n+1}({\Bbb C}(0,1;2)))$.
On the other hand, again by Theorem \ref{AllC}, $\pi_1^{orb}(B_{n+1}({\Bbb C}(0,1;2)))$ contains
the Artin group of type $D_{n+1}$ as a subgroup of index $2$. Therefore, using Condition $(2)$ and the
above $D_n$-case, we
complete the proof. The flow of the argument is given in the following diagram.

\medskip
\centerline{
  \xymatrix{&&0\\
      &&{\Bbb Z}_2\ar[u]\\
      {\cal A}_{\tilde B_n}\ar@{^{(}->}[r]^{\hspace{-11mm} {\rotatebox{90}{\text{Theorem 4.1}}}}&
      \pi_1^{orb}(B_n({\Bbb C}(1,1;2)))\ar@{^{(}->}[r]^{\hspace{-1mm}{\rotatebox{90}{\text{Proposition 4.1}}}}&
      \pi_1^{orb}(B_{n+1}({\Bbb C}(0,1;2)))\ar[u]\\
  &&{\cal A}_{D_{n+1}}\ar@{^{(}->}[u]^{\text{Theorem 4.1}}}}
\end{proof}

\medskip
\section{Computation of the surgery groups}
This section is devoted to some application related to computation of the 
surgery groups of some of the discrete groups considered in this paper. 

Since the finite type pure Artin groups are torsion free, Corollary \ref{lkt} is 
applicable to the finite type pure Artin groups considered in Theorem \ref{fic-braid}. 
A parallel to this $K$-theoretic  
vanishing result is the computation of the surgery groups. 

For finite groups, the computation of the surgery groups is  
well established (\cite{HT}). The infinite groups case needs different 
techniques and is difficult, even 
when we have the isomorphism of the $L$-theory assembly map. For a 
survey on known results and techniques, on computation of surgery groups 
for infinite groups, see \cite{R1}.

In the case of classical pure braid group, we did the computation 
in \cite{R5}. Here we extend it to the pure Artin groups 
of the finite type Artin groups considered in Theorem \ref{fic-braid}.
 
The main idea behind 
the computation is the following lemma on the homotopy type of the first suspension of a 
hyperplane arrangement complement. This lemma was stated and proved in 
\cite{R5}.
 
\begin{lemma}\label{suspension} (\cite{R5}, Lemma 4.1) The first suspension $\Sigma (
{\Bbb C}^n-\cup_{j=1}^N A_j)$ of the complement of a   
hyperplane arrangement ${\cal A}=\{A_1,A_2,\ldots , A_N\}$ 
in ${\Bbb C}^n$, is 
homotopically equivalent to the wedge of spheres 
$\vee_{j=1}^N S_j$, where $S_j$ is 
homeomorphic to the $2$-sphere ${\Bbb S}^2$ for $j=
1,2, \ldots , N$.\end{lemma}

We need the following result to prove Theorem \ref{surgery}.

\begin{prop}\label{mainthm2}
Let ${\cal A}=\{A_1, A_2,\ldots , A_N\}$ be an   
aspherical hyperplane arrangement in ${\Bbb C}^n$. Assume that the two 
assembly maps $A_K$ and $A_L$ are isomorphisms, for the group 
$\G=\pi_1({\Bbb C}^n-\cup_{j=1}^{N}{A_j})$. Then the surgery groups of 
$\G$ are  
given by the following.

$$
L_i(\G) =\begin{cases} {\Bbb Z}&\text{if
$i\equiv$ 0 mod 4},\\
{\Bbb Z}^N &\text{if $i\equiv$ 1 mod 4},\\

{\Bbb Z}_2&\text{if $i\equiv$ 2 mod 4},\\

{\Bbb Z}_2^N &\text{if $i\equiv$ 3 mod 4}.
\end{cases}$$
\end{prop}

\begin{proof} 
  The proposition was stated in [\cite{R5}, Theorem 2.2] for
  fiber-type hyperplane arrangement complements. But the proof of this
  more general statement is the same.
  First, by Lemma \ref{suspension},  
the first suspension of ${\Bbb C}^n-\cup_{j=1}^{N}{A_j}=X$ (say) is homotopically equivalent to 
the wedge of $N$ many $2$-spheres. Hence, for any generalized homology theory 
$h_*$, applied over $X$, we get the following equation. 
$$h_i(X)=h_i(*)\oplus h_{i-2}(*)^N.$$ 

Where $*$ denotes a single point space. The proof is  
completed using the known computation of the surgery 
groups of the trivial group, and since $A_L$ is an isomorphism.\end{proof}

\begin{proof}[Proof of Theorem \ref{surgery}] Note that, the pure Artin groups of 
the finite type Artin groups are torsion free, since they are fundamental 
groups of finite dimensional aspherical hyperplane arrangement complements 
in some complex space. See Remark \ref{torsion-free}. Hence, by
Theorem \ref{fic-braid}  
the assembly map $A_L$ is an isomorphism for these hyperplane arrangement complements. 

\medskip
\begin{tabular}{|l|l|l|}
\hline
${\cal {PA}}=$&&\\
{\bf pure Artin} &{\bf Reflecting hyperplanes  in} &{\bf Number of}\\
{\bf group type}&${\Bbb C}^n=\{(z_1,z_2,\ldots ,z_n)\ |\ z_i\in {\Bbb C}\}$&{\bf reflecting}\\
&&{\bf hyperplanes}\\
  \hline \hline
  $A_{n-1}$&$z_i=z_j$ for $i\neq j$.&$\frac{n(n-1)}{2}$\\
\hline
$B_n(=C_n)$ & $z_i=0$ for all $i$; $z_i=z_j$, $z_i=-z_j$ for $i\neq j$.&$n^2$\\
\hline
$D_n$&$z_i=z_j$, $z_i=-z_j$ for $i\neq j$.&$n(n-1)$\\ 
\hline
$F_4$&$n=4$; $z_i=0$ for all $i$; $z_i=z_j$, $z_i=-z_j$&$24$\\ 
&for $i\neq j$; $z_1\pm z_2\pm z_3\pm z_4=0$.  &\\
\hline
$I_2(p)$&$n=2$; roots are in one-to-one  &$p$\\
&correspondence with the lines&\\
&of symmetries of a regular $p$-gon in ${\Bbb R}^2$.&\\
& See [\cite{Hum}, page 4]. Hence $p$ hyperplanes.&\\
\hline
$G_2=I_2(6)$& &$6$\\
\hline
\end{tabular}

\medskip

\centerline{Table 5}

Therefore, by Proposition \ref{mainthm2} together 
with the calculation in Table 5 of the number of hyperplanes associated 
to the finite reflection groups, we complete the proof of the corollary.
We refer the reader to \cite{Bri}, for the 
equations of the hyperplanes given in the table. Or 
see [\cite{Hum}, p. 41-43] for a complete root structure.
 
\end{proof}

We end with the following remark and a problem.

\begin{rem}\label{asphericity}{\rm The asphericity  
    of certain hyperplane arrangement complements, including the configuration space of aspherical
    $2$-manifolds, was needed in this work,
    to conclude that they have torsion free
    fundamental groups, and also for some induction argument to work
    (see Lemma \ref{induction}). This result is well 
    known for configuration spaces of aspherical
    $2$-manifolds (see Theorem \ref{fn} and \cite{Bri}). It is known that the 
finite type real Artin groups are fundamental groups 
of aspherical manifolds (\cite{Del}). In the affine type 
Artin group case, this is proved recently in \cite{PS}.

For the configuration space of
    aspherical $2$-dimensional good orbifold, we state the following Problem.}\end{rem}

\medskip
\noindent
{\bf Problem.} Show that the orbifold braid group of any $2$-dimensional good orbifold, whose
universal cover is contractible, belongs to $\cal C$ or to the class
of groups defined in [\cite{lueck}, Theorem 20.8.6]. In particular, this will prove that the
FICwF is true for the affine Artin group of type $\tilde D_n$. See Theorem \ref{AllC}. Note that,
the pure braid space of such a $2$-dimensional good orbifold is again a good orbifold. But it is
not yet known if the universal cover of this pure braid space is contractible. Interesting
simple case to prove this will be ${\Bbb C}(0,2;2,2)$.

\newpage
\bibliographystyle{plain}
\ifx\undefined\bysame
\newcommand{\bysame}{\leavevmode\hbox to3em{\hrulefill}\,}
\fi

\end{document}